\documentclass[10pt]{book}
\usepackage[cp1251]{inputenc} \usepackage[T2A]{fontenc}
\usepackage[english,russian,ukrainian]{babel}
\usepackage{amssymb} \usepackage{amsthm} \usepackage{amsmath}
\usepackage{amsmath,amsthm,amsfonts,amssymb,mathrsfs}

\input xymatrix
\xyoption{all}

\newtheorem{theorem}{Теорема}

\newtheorem{proposition}[theorem]{Твердження}
\newtheorem{corollary}[theorem]{Наслiдок}
\newtheorem{lemma}[theorem]{Лема}
\theoremstyle{definition}
\newtheorem{example}[theorem]{Приклад}
\newtheorem{remark}[theorem]{Зауваження}

\newtheorem{definition}[theorem]{Означення}

\begin{document}
\setcounter{equation}{0}
\large
\twocolumn[УДК 512.534
\begin{center}
{\Large \bf \copyright 2018 р. \,\, Олег Гутік, Анатолій Савчук \/}
\bigskip

Львівський національний університет імені Івана Франка
\end{center}

\begin{center}
{\bf НАПІВГРУПА ЧАСТКОВИХ КОСКІНЧЕННИХ ІЗОМЕТРІЙ НАТУРАЛЬНИХ ЧИСЕЛ }
\end{center}

\begin{center}
\parbox{15cm}{\normalsize \parindent=0.5cm
Вивчається напівгрупа $\mathbf{I}\mathbb{N}_{\infty}$ усіх часткових коскінченних ізометрій множини натуральних чисел. Ми описуємо відношення Ґріна на напівгрупі $\mathbf{I}\mathbb{N}_{\infty}$, її в'язку та доводимо, що $\mathbf{I}\mathbb{N}_{\infty}$~--- проста $E$-унітарна $F$-інверсна напівгрупа. Описана найменша групова конгруенція $\mathfrak{C}_{\mathbf{mg}}$ на напівгрупі $\mathbf{I}\mathbb{N}_{\infty}$ та доведено, що фактор-напівгрупа $\mathbf{I}\mathbb{N}_{\infty}/\mathfrak{C}_{\mathbf{mg}}$ ізоморфна адитивній групі цілих чисел. Наведено приклад конгруенції на напівгрупі $\mathbf{I}\mathbb{N}_{\infty}$, яка не є груповою. Також доведено, що конгруенція на $\mathbf{I}\mathbb{N}_{\infty}$ є груповою тоді і лише тоді, коли її звуження на довільну піднапігрупу $S$ в $\mathbf{I}\mathbb{N}_{\infty}$, яка ізоморфна біциклічній напівгрупі, є груповою конгруенцією на $S$.

\bigskip

\textbf{Oleg Gutik, Anatolii Savchuk, The semigroup of partial co-finite isometries of positive integers}.

The semigroup $\mathbf{I}\mathbb{N}_{\infty}$ of all partial co-finite isometries of positive integers is studied. We describe Green's relations on the semigroup $\mathbf{I}\mathbb{N}_{\infty}$, its band and proved that $\mathbf{I}\mathbb{N}_{\infty}$ is a simple $E$-unitary $F$-inverse semigroup. We described the least group congruence $\mathfrak{C}_{\mathbf{mg}}$ on $\mathbf{I}\mathbb{N}_{\infty}$ and proved that the quotient-semigroup  $\mathbf{I}\mathbb{N}_{\infty}/\mathfrak{C}_{\mathbf{mg}}$ is isomorphic to the additive group of integers. An example of a non-group congruence on the semigroup $\mathbf{I}\mathbb{N}_{\infty}$ is presented. Also we proved that a congruence on the semigroup $\mathbf{I}\mathbb{N}_{\infty}$ is a group congruence if and only if its restriction onto an isomorphic  copy of the bicyclic semigroup in $\mathbf{I}\mathbb{N}_{\infty}$ is a group congruence.}
\end{center}

\bigskip
] \markboth{}{}

У даній праці ми користуватимемося термінологією з [7, 9, 11].
Надалі у тексті множину натуральних чисел  позначатимемо через $\mathbb{N}$.

Якщо визначене часткове відображення $\alpha\colon X\rightharpoonup Y$ з множини $X$ у множину $Y$, то через $\operatorname{dom}\alpha$ i $\operatorname{ran}\alpha$ будемо позначати його \emph{область визначення} та \emph{область значень}, відповідно, а через $(x)\alpha$ та $(A)\alpha$ --- образи елемента $x\in\operatorname{dom}\alpha$ та підмножини $A\subseteq\operatorname{dom}\alpha$ при частковому відображенні $\alpha$, відповідно. Часткове відображення $\alpha\colon X\rightharpoonup Y$ називається \emph{ко-скінченним}, якщо множини $X\setminus\operatorname{dom}\alpha$ та $Y\setminus\operatorname{ran}\alpha$ є скінченними.

Через $\mathscr{I}_\lambda$ позначимо множину усіх часткових взаємно однозначних перетворень множини $X$ потужності $\lambda$ разом з такою напівгруповою операцією:
$
(x)(\alpha\beta)=((x)\alpha)\beta,
$
якщо
$x\in\operatorname{dom}(\alpha\beta)=\left\{
y\in\operatorname{dom}\alpha\colon
(y)\alpha\in\operatorname{dom}\beta\right\}$, для
$\alpha,\beta\in\mathscr{I}_\lambda$.
Напівгрупа $\mathscr{I}_\lambda$ називається  \emph{симеричним ін\-верс\-ним моноїдом} (або \emph{симетричною ін\-верс\-ною напівгрупою}) над множиною $X$~(див [7, \S1.9].
Симетрична інверсна на\-пів\-гру\-па вперше введена В.В.~Вагнером у праці~[2] і вона відіграє дуже важливу роль в алгебраїчній теорії  напівгруп.

Рефлексивне, антисиметричне та транзитивне відношення на множині $X$ назива\-єть\-ся \emph{частковим порядком} на $X$. Множина $X$ із заданим на ній частковим порядком $\leqslant$ називається \emph{частково впорядкованою множиною} і позначається $(X,\leqslant)$.

Елемент $x$ частково впорядкованої множини $(X,\leqslant)$ називається \emph{найбільшим} (\emph{найменшим}) в $(X,\leqslant)$, якщо $y\leqslant x$ ($x\leqslant y$) для всіх $y\in X$.

У випадку, якщо $(X,\leqslant)$~--- частково впорядкована множина й $x\leqslant y$, для деяких $x,y\in X$, то будемо говорити, що елементи $x$ і $y$~є \emph{порівняльними} в $(X,\leqslant)$. Якщо ж для елементів $x, y$ частково впорядкованої множини $(X,\leqslant)$ не виконується жодне з відношень $x\leqslant y$ або $y\leqslant x$, то говоритимемо, що елементи $x$ і $y$ є \emph{непорівняльними} в $(X,\leqslant)$. Частковий порядок $\leqslant$ на $X$ називається \emph{лінійним}, якщо довільні два елементи в $(X,\leqslant)$ є порівняльними. У цьому випадку ми будемо говорити, що $(X,\leqslant)$ є \emph{лінійно впорядкованою множиною} або \emph{ланцюгом}.

Відображення $h\colon X\rightarrow Y$ з частково впо\-ряд\-кованої множини $(X,\leqslant)$ в частково впорядковану множину $(Y,\leqslant)$ назива\-єть\-ся \emph{монотонним}, якщо з $x\leqslant y$ випливає $(x)h\leqslant (y)h$. Монотонне бієктивне відображення $h\colon (X,\leqslant)\rightarrow (Y,\leqslant)$ частково впорядкованих множин, обернене до якого є монотонним, називається \emph{порядковим ізо\-мор\-фіз\-мом}. Лінійно впорядкована множина, яка порядково ізоморфна $(\mathbb{N},\geqslant)$ називається \emph{$\omega$-ланцюгом}.

Якщо $S$~--- напівгрупа, то її підмножина ідемпотентів позначається через $E(S)$.  На\-пів\-гру\-па $S$ називається \emph{інверсною}, якщо для довільного її елемента $x$ існує єдиний елемент $x^{-1}\in S$ такий, що $xx^{-1}x=x$ та $x^{-1}xx^{-1}=x^{-1}$ [2]. В інверсній напівгрупі $S$ вище означений елемент $x^{-1}$ називається \emph{інверсним до} $x$. \emph{В'язка}~--- це напівгрупа ідемпотентів, а \emph{напівгратка}~--- це комутативна в'язка. Надалі через $(\mathscr{P}_{\!\infty}(X),\cup)$ по\-зна\-ча\-ти\-ме\-мо \emph{вільну напівгратку} з одиницею над не\-порож\-ньою множиною $X$, тобто множину усіх скінченних (включно з по\-рож\-ньою) підмножин множини $X$ з операцією ``об'єднання''.

Якщо $S$ --- напівгрупа, то ми позначатимемо відношення Ґріна на $S$ через $\mathscr{R}$, $\mathscr{L}$, $\mathscr{D}$, $\mathscr{H}$ і $\mathscr{J}$ (див. означення в [7, \S2.1]. Напівгрупа $S$ називається \emph{простою}, якщо $S$ не містить власних двобічних ідеалів, тобто $S$ складається з одного $\mathscr{J}$-класу.

Відношення еквівалентності $\mathfrak{K}$ на напівгрупі $S$ називається \emph{конгруенцією}, якщо для елементів $a$ та $b$ напівгрупи $S$ з того, що виконується умова $(a,b)\in\mathfrak{K}$ випливає, що $(ca,cb), (ad,bd) \in\mathfrak{K}$, для всіх $c,d\in S$. Відношення $(a,b)\in\mathfrak{K}$ ми також будемо записувати $a\mathfrak{K}b$, і в цьому випадку будемо говорити, що \emph{елементи $a$ i $b$ є $\mathfrak{K}$-еквівалентними}.

Якщо $S$~--- напівгрупа, то на $E(S)$ визначено частковий порядок:
$
e\preccurlyeq f
$   тоді і лише тоді, коли
$ef=fe=e$.
Так означений частковий порядок на $E(S)$ називається \emph{при\-род\-ним}.

Означимо відношення $\preccurlyeq$ на інверсній напівгрупі $S$ так:
$
    s\preccurlyeq t
$
тоді і лише тоді, коли $s=te$.
для деякого ідемпотента $e\in S$. Так означений частковий порядок назива\-єть\-ся \emph{при\-род\-ним част\-ковим порядком} на інверсній напівгрупі $S$~[9]. Очевидно, що звуження природного часткового порядку $\preccurlyeq$ на інверсній напівгрупі $S$ на її в'язку $E(S)$ є при\-род\-ним частковим порядком на $E(S)$.

Часткове перетворення $\alpha\colon (X,d)\rightharpoonup (X,d)$ метричного простору $(X,d)$ на\-зи\-ва\-єть\-ся \emph{ізо\-мет\-рич\-ним} або \emph{частковою ізометрією}, якщо $d(x\alpha,y\alpha)=d(x,y)$ для довільних $x,y\in (X,d)$. Очевидно, що композиція двох часткових ізометрій метричного прос\-то\-ру $(X,d)$ є знову част\-ко\-вою ізометрією, а також, що обернене часткове відображення до часткової ізо\-мет\-рії є частковою ізометрією. Таким чином, часткові ізометрії метричного простору $(X,d)$ стосовно операції композиції спсткових перетворень є інверсним підмоноїдом симет\-рич\-ного інверсного моноїда над множиною $X$.

Напівгрупа $\mathbf{ID}_{\infty}$ усіх часткових коскінченних ізометрій множини цілих чисел $\mathbb{Z}$ означена в праці Безущак [6], де описані її твірні та доведено, що вона має експоненціальний ріст. Зауважимо, що напівгрупа $\mathbf{ID}_{\infty}$ є інверсною і є, очевидно, піднапівгрупою напівгрупи всіх часткових коскінченних бієкцій множини цілих чисел $\mathbb{Z}$, а елементи напівгрупи $\mathbf{ID}_{\infty}$ --- це саме звуження ізометрій множини цілих чисел $\mathbb{Z}$ на коскінченні підмножини в розумінні Лоусона (див. [9, c. 9]. У праці [1] описані відношення Ґріна та головні ідеали напівгрупи $\mathbf{ID}_{\infty}$. У [3] доведено, що фактор-напівгрупа $\mathbf{ID}_{\infty}/\mathfrak{C}_{\mathbf{mg}}$ за мінімальною груповою конгруенцією $\mathfrak{C}_{\mathbf{mg}}$ ізоморфна групі ${\mathbf{Iso}}(\mathbb{Z})$ усіх ізометрій множини $\mathbb{Z}$, напівгрупа $\mathbf{ID}_{\infty}$ є $F$-інверсною напівгрупою, а також, що напівгрупа $\mathbf{ID}_{\infty}$ ізоморфна напівпрямому до\-бут\-ку ${\mathbf{Iso}}(\mathbb{Z})\ltimes_\mathfrak{h}\mathscr{P}_{\!\infty}(\mathbb{Z})$ вільної напівгратки з одиницею $(\mathscr{P}_{\!\infty}(\mathbb{Z}),\cup)$ групою ${\mathbf{Iso}}(\mathbb{Z})$. Також, у [3] досліджувалась топологізація напівгрупи $\mathbf{ID}_{\infty}$ та задача ізоморф\-ного занурення дискретної напівгрупи $\mathbf{ID}_{\infty}$ у гаусдорфові топологічні напівгрупи близькі до ком\-пакт\-них.

Нехай $\mathbf{I}\mathbb{N}_{\infty}$ --- множина усіх часткових коскінченних ізометрій множини натуральних чисел $\mathbb{N}$ зі звичайною метрикою $d(n,m)=|n-m|$, $n,m\in \mathbb{N}$. Оскільки множина $\mathbf{I}\mathbb{N}_{\infty}$ замк\-не\-на стосовно операції композиції част\-ко\-вих відображень та взяття оберненого част\-ко\-во\-го відображення, то $\mathbf{I}\mathbb{N}_{\infty}$ --- інверсний підмоноїд симетричного інверсного моноїда $\mathscr{I}_\omega$. Через $\mathbb{I}$ позначатимемо тотожне ві\-доб\-ра\-жен\-ня множини натуральних чисел $\mathbb{N}$. Очевидно, що $\mathbb{I}$~--- одиниця моноїда $\mathbf{I}\mathbb{N}_{\infty}$.


У цій праці ми досліджуємо алгебраїчні властивості напівгрупи $\mathbf{I}\mathbb{N}_{\infty}$. Зокрема, описуємо відношення Ґріна на напівгрупі $\mathbf{I}\mathbb{N}_{\infty}$, її в'язку та доводимо, що $\mathbf{I}\mathbb{N}_{\infty}$ --- проста $E$-унітарна $F$-інверсна напівгрупа. Описана найменша групова конгруенція $\mathfrak{C}_{\mathbf{mg}}$ на напівгрупі $\mathbf{I}\mathbb{N}_{\infty}$ та доведено, що фактор-напівгрупа $\mathbf{I}\mathbb{N}_{\infty}/\mathfrak{C}_{\mathbf{mg}}$ ізоморфна адитивній групі цілих чисел. Наведено приклад конгруенції на напівгрупі $\mathbf{I}\mathbb{N}_{\infty}$, яка не є груповою. Також доведено, що конгруенція на $\mathbf{I}\mathbb{N}_{\infty}$ є груповою тоді і лише тоді, коли її звуження на довільну піднапігрупу $S$ в $\mathbf{I}\mathbb{N}_{\infty}$, яка ізоморфна біциклічній напівгрупі, є груповою конгруенцією на $S$.


Нехай $\mathbb{Z}$~--- множина цілих чисел. Ві\-доб\-ра\-жен\-ня $f\colon\mathbb{Z}\to\mathbb{Z}$, означене за формулою $(z)f=z+z_0$, де $z_0$~--- деяке ціле число будемо називати \emph{зсувом множини цілих чисел}.  Часткове відображення $\alpha\colon \mathbb{N}\rightharpoonup \mathbb{N}$ на\-зи\-ва\-єть\-ся
\emph{звуженням часткового зсуву множини натуральних чисел}, якщо існують зсув множини цілих чисел $f\colon\mathbb{Z}\to\mathbb{Z}$ і $A\subseteq \mathbb{N}$ такі, що $\operatorname{dom}\alpha=A$ i $(x)f=(x)\alpha$, для всіх $x\in\operatorname{dom}\alpha$. Якщо $\alpha\colon \mathbb{N}\rightharpoonup \mathbb{N}$~--- звуженням часткового зсуву множини натуральних чисел i $B\subseteq\operatorname{dom}\alpha$, то про образ $(B)\alpha$ будемо називати \emph{зсувом} множини $B$.

\begin{lemma}\label{lemma-1}
Кожен елемент напівгрупи $\mathbf{I}\mathbb{N}_{\infty}$ є монотонною частковою бієкцією лінійно впорядкованої множини $(\mathbb{N},\leqslant)$. Більше того, кожен елемент напівгрупи $\mathbf{I}\mathbb{N}_{\infty}$ є звуженням часткового зсуву множини натуральних чисел на коскінченну підмножину в $\mathbb{N}$.
\end{lemma}

\begin{proof}
Зафіксуємо довільний елемент $\alpha$ напівгрупи $\mathbf{I}\mathbb{N}_{\infty}$. Позаяк $\alpha$~--- коскінченна часткова бієкція множини $\mathbb{N}$ і $(\mathbb{N},\leqslant)$~--- цілком впорядкована множина, то існує найменше натуральне число $n_{\alpha}$ таке, що $n\in\operatorname{dom}\alpha$ для всіх натуральних $n\geqslant n_{\alpha}$. Також, оскільки $\alpha$~--- часткова коскінченна ізо\-мет\-рія множини натуральних чисел, то
\begin{equation*}
  d((n_{\alpha}+1)\alpha,(n_{\alpha})\alpha)= \left|(n_{\alpha}+1)\alpha-(n_{\alpha})\alpha\right|=1,
\end{equation*}
а отже виконується одна з умов
\begin{equation*}
  (n_{\alpha}+1)\alpha=(n_{\alpha})\alpha+1 \; \hbox{або} \; (n_{\alpha}+1)\alpha=(n_{\alpha})\alpha-1.
\end{equation*}
Припустимо, що $(n_{\alpha}+1)\alpha=(n_{\alpha})\alpha-1$. Тоді за індукцією, оскільки $\alpha$~--- часткова коскінченна ізометрія множини натуральних чисел, то отримуємо, що $(n_{\alpha}+i)\alpha=(n_{\alpha})\alpha-i$ для довільного натурального числа $i$, що суперечить тому, що множина натуральних чисел має найменший елемент. З отриманого протиріччя випливає, що виконується рівність $(n_{\alpha}+1)\alpha=(n_{\alpha})\alpha+1$. Аналогічно, за індукцією, оскільки $\alpha$~--- часткова коскінченна ізометрія множини натуральних чисел, то отримуємо, що $(n_{\alpha}+i)\alpha=(n_{\alpha})\alpha+i$ для довільного натурального числа $i\geqslant2$. Також з вище доведеного випливає, що $(n)\alpha=(n_{\alpha})\alpha-n_{\alpha}+n$ для довільного $n\in\operatorname{dom}\alpha$, а отже виконується друге твердження леми.
\end{proof}

Через $\mathscr{I}_{\infty}^{\!\nearrow}(\mathbb{N})$ позначимо напівгрупу монотонних коскінченних часткових бієкцій множини натуральних чисел (див. [8]). Ос\-кіль\-ки існують монотонні коскінченні част\-ко\-ві бієкції множини натуральних чисел, які не є частковими ізометріями, то з леми~\ref{lemma-1} випливає

\begin{corollary}\label{corollary-2}
$\mathbf{I}\mathbb{N}_{\infty}$~--- власний підмоноїд в $\mathscr{I}_{\infty}^{\!\nearrow}(\mathbb{N})$.
\end{corollary}

\begin{proposition}\label{proposition-3}
\begin{itemize}
  \item[$(i)$] $E(\mathbf{I}\mathbb{N}_{\infty})=E(\mathscr{I}_{\infty}^{\!\nearrow}(\mathbb{N}))$ в $\mathscr{I}_{\infty}^{\!\nearrow}(\mathbb{N})$, а отже напівгратка $E(\mathbf{I}\mathbb{N}_{\infty})$ ізоморфна вільній напівгратці з одиницею $(\mathscr{P}_{\!\infty}(\mathbb{N}),\cup)$, і цей ізоморфізм визначається відображенням
      $(\varepsilon)h=\mathbb{N}\setminus\operatorname{dom}\varepsilon$.

  \item[$(ii)$] Якщо $\varepsilon,\iota\in E(\mathbf{I}\mathbb{N}_{\infty})$, то $\varepsilon\leqslant\iota$ тоді і лише тоді, коли           $\operatorname{dom}\varepsilon\subseteq\operatorname{dom}\iota$.

  \item[$(iii)$] Кожен максимальний ланцюг у напівгратці $E(\mathbf{I}\mathbb{N}_{\infty})$ є $\omega$-ланцюгом.

  \item[$(iv)$] $\alpha\mathscr{L}\beta$ в $\mathbf{I}\mathbb{N}_{\infty}$ тоді і лише тоді, коли $\operatorname{dom}\alpha=\operatorname{dom}\beta$.

  \item[$(v)$] $\alpha\mathscr{R}\beta$ в $\mathbf{I}\mathbb{N}_{\infty}$ тоді і лише тоді, коли $\operatorname{ran}\alpha=\operatorname{ran}\beta$.

  \item[$(vi)$] $\alpha\mathscr{H}\beta$ в  $\mathbf{I}\mathbb{N}_{\infty}$ тоді і лише тоді, коли $\alpha=\beta$.

  \item[$(vii)$] $\alpha\mathscr{D}\beta$ в  $\mathbf{I}\mathbb{N}_{\infty}$ тоді і лише тоді, коли $\operatorname{dom}\alpha$ $(\operatorname{ran}\alpha)$ є зсувом множини $\operatorname{dom}\beta$ $(\operatorname{ran}\beta)$.
\end{itemize}
\end{proposition}

\begin{proof}
$(i)$ Позаяк кожне часткове коскінченне тотожне перетворення множини натуральних чисел є частковою ізометрією, то за наслідком~\ref{corollary-2} маємо, що $E(\mathbf{I}\mathbb{N}_{\infty})=E(\mathscr{I}_{\infty}^{\!\nearrow}(\mathbb{N}))$ в $\mathscr{I}_{\infty}^{\!\nearrow}(\mathbb{N})$. Останнє твердження є наслідком тверд\-жен\-ня~2.1$(vii)$ з [8].

Твердження $(ii)$--$(vi)$ є наслідками тверд\-жен\-ня~2.1 з [8].

$(vii)$ Еквівалентність того, шо множина $\operatorname{dom}\alpha$ $(\operatorname{ran}\alpha)$ є зсувом множини $\operatorname{dom}\beta$ $(\operatorname{ran}\beta)$ випливає з другого твердження леми~\ref{lemma-1}.

За означенням відношення  Ґріна $\mathscr{D}$ маємо, що $\mathscr{D}=\mathscr{L}\circ\mathscr{R}=\mathscr{R}\circ\mathscr{L}=\mathscr{R}\bigvee\mathscr{L}$, і оскільки $\alpha\mathscr{L}\alpha\alpha^{-1}$ i $\beta\mathscr{R}\beta^{-1}\beta$, то $\alpha\mathscr{D}\beta$ в  $\mathbf{I}\mathbb{N}_{\infty}$ тоді і лише тоді, коли $\alpha\alpha^{-1}\mathscr{D}\beta^{-1}\beta$,. Тоді за твердженням~3.2.5 з [9] існує елемент $\gamma\in\mathbf{I}\mathbb{N}_{\infty}$ такий, що $\gamma\gamma^{-1}=\alpha\alpha^{-1}$ i $\gamma^{-1}\gamma=\beta^{-1}\beta$. Останні дві рівності виконуються тоді і лише тоді, коли $\operatorname{dom}\alpha=\operatorname{dom}\gamma$ і $\operatorname{ran}\gamma=\operatorname{ran}\beta$, а отже умова $\alpha\mathscr{D}\beta$ еквівалентна умові, що $\operatorname{dom}\alpha$ є зсувом множини $\operatorname{ran}\beta$.
\end{proof}

\begin{theorem}\label{theorem-4}
$\mathbf{I}\mathbb{N}_{\infty}$~--- проста напівгрупа.
\end{theorem}

\begin{proof}
Оскільки $\alpha=\alpha\mathbb{I}=\mathbb{I}\alpha$ для довільного елемента $\alpha$ напівгрупи $\mathbf{I}\mathbb{N}_{\infty}$, то нам достатньо довести, що для довільного елемента $\beta$ напівгрупи $\mathbf{I}\mathbb{N}_{\infty}$ існують $\gamma,\delta\in\mathbf{I}\mathbb{N}_{\infty}$ такі, що $\gamma\beta\delta=\mathbb{I}$.

Зафіксуємо довільний елемент $\beta$ напівгрупи $\mathbf{I}\mathbb{N}_{\infty}$. Оскільки за лемою~\ref{lemma-1} кожен елемент напівгрупи $\mathbf{I}\mathbb{N}_{\infty}$ є звуженням часткового зсуву множини натуральних чисел на коскінченну підмножину в $\mathbb{N}$ і $(\mathbb{N},\leqslant)$~--- цілком впорядкована множина, то існує найменше натуральне число $n_{\beta}^{\mathbf{d}}\in\operatorname{dom}\beta$ таке, що $n\in\operatorname{dom}\beta$ для всіх натуральних $n\geqslant n_{\beta}^{\mathbf{d}}$ та існує найменше натуральне число $n_{\beta}^{\mathbf{r}}\in\operatorname{ran}\beta$ таке, що $n\in\operatorname{dom}\beta$ для всіх натуральних $n\geqslant n_{\beta}^{\mathbf{r}}$. Покладемо
$$
\operatorname{dom}\gamma=\mathbb{N}, \qquad \operatorname{ran}\gamma=\left\{n\in\mathbb{N}\colon n\geqslant n_{\beta}^{\mathbf{d}}\right\},
$$
$$(i)\gamma=i-1+n_{\beta}^{\mathbf{d}} \qquad \hbox{для всіх} \; i\in\operatorname{dom}\gamma
$$
i
$$
\operatorname{dom}\delta=\left\{n\in\mathbb{N}\colon n\geqslant n_{\beta}^{\mathbf{r}}\right\}, \quad \operatorname{ran}\delta=\mathbb{N},
$$
$$
(j)\delta=i-n_{\beta}^{\mathbf{r}}+1 \qquad \hbox{для всіх}\;  j\in\operatorname{dom}\delta.
$$
Тоді з леми~\ref{lemma-1} випливає, що $\gamma\beta\delta=\mathbb{I}$.
\end{proof}

Наступне очевидне твердження описує природний частковий порядок на напівгрупі $\mathbf{I}\mathbb{N}_{\infty}$ і воно випливає з описання природного часткового порядку на симетричному інверсному моноїді, оскільки $\mathbf{I}\mathbb{N}_{\infty}$ є інверсним підмоноїдом симетричного інверсного моноїда $\mathscr{I}_\omega$ над множиною натуральних чисел $\mathbb{N}$.

\begin{proposition}\label{proposition-5}
Для елементів $\alpha$ i $\beta$ напівгрупи $\mathbf{I}\mathbb{N}_{\infty}$ такі умови є еквівалентними:
\begin{itemize}
  \item[$(i)$] $\alpha\preccurlyeq\beta$ в $\mathbf{I}\mathbb{N}_{\infty}$;
  \item[$(ii)$] часткове відображення $\alpha$ є звуженням часткового відображення $\beta$ на $\operatorname{dom}\alpha$;
  \item[$(iii)$] часткове відображення $\alpha$ є ко\-зву\-жен\-ням\footnote{Нехай $\alpha\colon X\rightharpoonup Y$~--- часткове відображення та $B$~--- підмножина в $Y$. Під козвуженням часткове відображення $\alpha$ будемо розуміти часткове відображення $\alpha{\downharpoonright}_{B}\colon X\rightharpoonup Y$ з $\operatorname{dom}\alpha{\downharpoonright}_{B}=\{x\in X\colon (x)\alpha\in B\}$ i $\operatorname{ran}\alpha{\downharpoonright}_{B}=B$.} часткового відображення $\beta$ на $\operatorname{ran}\alpha$.
\end{itemize}
\end{proposition}

Нагадаємо [9], шо інверсна напівгрупа $S$ називається \emph{$E$-унітарною}, якщо $ex$~-- ідемпотент в $S$ для деякого ідемпотента $e\in S$ та $x\in S$, то $x$~-- ідемпотент напівгрупи $S$. Тоді з твердження~\ref{proposition-5} i другої частини леми~\ref{lemma-1} випливає:

\begin{corollary}\label{corollary-6}
$\mathbf{I}\mathbb{N}_{\infty}$~--- $E$-унітарна інверсна напівгрупа.
\end{corollary}

\emph{Найменша групова конгруенція} $\mathfrak{C}_{\mathbf{mg}}$ на інверсній напівгрупі $S$ визначається так (див. [11, III.5]:
$s\mathfrak{C}_{\mathbf{mg}}t$  в $S$ тоді і лише тоді, коли існує ідемпотент $e\in S$  такий, що $es=et$.

Наступне твердження описує найменшу гру\-пову конгруенцію $\mathfrak{C}_{\mathbf{mg}}$ на напівгрупі~$\mathbf{I}\mathbb{N}_{\infty}$.

\begin{proposition}\label{proposition-7}
Для елементів $\alpha$ та $\beta$ напівгрупи $\mathbf{I}\mathbb{N}_{\infty}$ такі умови є еквівалентними:
\begin{itemize}
  \item[$(i)$] $\alpha\mathfrak{C}_{\mathbf{mg}}\beta$;
  \item[$(ii)$] існує натуральне число $i\in\operatorname{dom}\alpha\cap\operatorname{dom}\beta$ таке, що $(i)\alpha=(i)\beta$;
  \item[$(iii)$] $(i)\alpha=(i)\beta$ для всіх $i\in\operatorname{dom}\alpha\cap\operatorname{dom}\beta$.
\end{itemize}
\end{proposition}

\begin{proof}
Імплікація $(iii)\Rightarrow(ii)$ очевидна. Оскільки за лемою~\ref{lemma-1} кожен елемент напівгрупи $\mathbf{I}\mathbb{N}_{\infty}$ є звуженням часткового зсуву множини натуральних чисел $\mathbb{N}$, то $(ii)\Rightarrow(iii)$.

$(iii)\Rightarrow(i)$
Якщо $i\alpha=i\beta$ для всіх $i\in\operatorname{dom}\alpha\cap\operatorname{dom}\beta$, то поклавши $\varepsilon\colon \operatorname{dom}\alpha\cap\operatorname{dom}\beta\to \operatorname{dom}\alpha\cap\operatorname{dom}\beta$~--- тотожне відображення, отримуємо, що $\varepsilon\in E(\mathbf{I}\mathbb{N}_{\infty})$ i $\varepsilon\alpha=\varepsilon\beta$.

$(i)\Rightarrow(ii)$ Припустимо, що $\varepsilon\alpha=\varepsilon\beta$ для деякого ідемпотента $\varepsilon$ напівгрупи $\mathbf{I}\mathbb{N}_{\infty}$. Оскільки $\varepsilon$~--- тотожне відображення коскінченної підмножини $\operatorname{dom}\varepsilon$ множини натуральних чисел, то з означення напівгрупи $\mathbf{I}\mathbb{N}_{\infty}$ випливає, що $\operatorname{dom}\alpha\cap\operatorname{dom}\beta\cap\operatorname{dom}\varepsilon\neq\varnothing$, а також, що
$(i)\alpha=(i)\varepsilon\alpha=(i)\varepsilon\beta=(i)\beta$
для всіх $i\in\operatorname{dom}\alpha\cap\operatorname{dom}\beta\cap\operatorname{dom}\varepsilon$.
\end{proof}

Надвлі в цій праці через $(\mathbb{Z},+)$ позначатимемо адитивну групу цілих чисел.

\begin{theorem}\label{theorem-8}
Фактор-напівгрупа $\mathbf{I}\mathbb{N}_{\infty}/\mathfrak{C}_{\mathbf{mg}}$ ізоморфна групі $(\mathbb{Z},+)$.
\end{theorem}

\begin{proof}
Зафіксуємо довільні елементи $\alpha$ та $\beta$ напівгрупи $\mathbf{I}\mathbb{N}_{\infty}$ такі, що $\alpha\mathfrak{C}_{\mathbf{mg}}\beta$. Оскільки за лемою~\ref{lemma-1} кожен елемент напівгрупи $\mathbf{I}\mathbb{N}_{\infty}$ є звуженням часткового зсуву множини натуральних чисел $\mathbb{N}$, то існують цілі числа $\mathbf{z}_{\alpha}$ та $\mathbf{z}_{\beta}$ такі, що $(i)\alpha=i+\mathbf{z}_{\alpha}$ і $(j)\beta=j+\mathbf{z}_{\beta}$ для довільних $i\in\operatorname{dom}\alpha$ та $j\in\operatorname{dom}\beta$. З твердження~\ref{proposition-7} випливає, що $\mathbf{z}_{\alpha}=\mathbf{z}_{\beta}$. Очевидно, з твердження~\ref{proposition-7} випливає, що виконується також обернене тверд\-жен\-ня: \emph{якщо $\mathbf{z}_{\alpha}=\mathbf{z}_{\beta}$ для елементів $\alpha$ та $\beta$ напівгрупи $\mathbf{I}\mathbb{N}_{\infty}$, то $\alpha\mathfrak{C}_{\mathbf{mg}}\beta$}.

Означимо відображення $\mathfrak{F}\colon \mathbf{I}\mathbb{N}_{\infty}\to (\mathbb{Z},+)$ за формулою $(\alpha)\mathfrak{F}=\mathbf{z}_{\alpha}$. З леми~\ref{lemma-1} і твердження~\ref{proposition-7} випливає, що відображення $\mathfrak{F}\colon \mathbf{I}\mathbb{N}_{\infty}\to (\mathbb{Z},+)$ визначено коректно. Тоді для довільних елементів $\alpha$ та $\beta$ напівгрупи $\mathbf{I}\mathbb{N}_{\infty}$ маємо, що
$ (i)\alpha\beta=(i+\mathbf{z}_{\alpha})\beta=i+\mathbf{z}_{\alpha}+\mathbf{z}_{\beta}$, для всіх  $i\in\operatorname{dom}(\alpha\beta)$.
Таким чином, ми отримуємо, що $(\alpha\beta)\mathfrak{F}=\mathbf{z}_{\alpha}+\mathbf{z}_{\beta}=(\alpha)\mathfrak{F}+(\beta)\mathfrak{F}$, а отже $\mathfrak{F}\colon \mathbf{I}\mathbb{N}_{\infty}\to (\mathbb{Z},+)$~--- гомоморфізм. Очевидно, що так визначене відображення $\mathfrak{F}$ є сюр'єктивним.

Також, з твердження~\ref{proposition-7} випливає, що $(\alpha)\mathfrak{F}=(\beta)\mathfrak{F}$ тоді і лише тоді, коли $\alpha\mathfrak{C}_{\mathbf{mg}}\beta$. Таким чином, найменша групова конгруенція $\mathfrak{C}_{\mathbf{mg}}$ на інверсній напівгрупі $\mathbf{I}\mathbb{N}_{\infty}$ по\-роджує гомоморфізм $\mathfrak{F}\colon \mathbf{I}\mathbb{N}_{\infty}\to (\mathbb{Z},+)$ і конгруенція $\mathfrak{C}_{\mathbf{mg}}$ є ядром цього гомоморфізму, а отже виконується твердження теореми.
\end{proof}

Нагадаємо, що інверсна напівгрупа $S$ називається \emph{$F$-інверсною}, якщо $\mathfrak{C}_{\textsf{mg}}$-клас $s_{\mathfrak{C}_{\textsf{mg}}}$ кожного елемента $s$  має найбільший елемент стосовно природного часткового порядку $\preccurlyeq$ в $S$ [10].

\begin{theorem}\label{theorem-9}
$\mathbf{I}\mathbb{N}_{\infty}$~--- $F$-інверсна напівгрупа.
\end{theorem}

\begin{proof}
Зафіксуємо довільний елемент $\alpha$ напівгрупи $\mathbf{I}\mathbb{N}_{\infty}$. Означимо часткове відображення $\alpha_\mathbf{m}\colon\mathbb{N}\rightharpoonup\mathbb{N}$ наступним чином. Нехай $\mathbf{z}_{\alpha}$~--- ціле число, означене для елемента $\alpha$ в тексті доведення теореми~\ref{theorem-8}. Тоді можливі такі випадки:
\begin{equation*}
  (i)~\mathbf{z}_{\alpha}<0; \quad (ii)~\mathbf{z}_{\alpha}=0 \quad \hbox{або} \quad (iii)~\mathbf{z}_{\alpha}>0.
\end{equation*}
Покладемо:
\begin{itemize}
  \item[$(i)$] якщо $\mathbf{z}_{\alpha}<0$, то $$\operatorname{dom}\alpha_\mathbf{m}=\left\{i\in\mathbb{N}\colon i\geqslant-\mathbf{z}_{\alpha}+1\right\},$$ $\operatorname{ran}\alpha_\mathbf{m}=\mathbb{N}$ i $(j)\alpha_\mathbf{m}=j+\mathbf{z}_{\alpha}$ для всіх $j\in \operatorname{dom}\alpha_\mathbf{m}$;

  \item[$(ii)$] якщо $\mathbf{z}_{\alpha}=0$, то $\alpha_\mathbf{m}\colon\mathbb{N}\to\mathbb{N}$~--- тотожне відображення;

  \item[$(iii)$] якщо $\mathbf{z}_{\alpha}>0$, то $\operatorname{dom}\alpha_\mathbf{m}=\mathbb{N}$, $$\operatorname{ran}\alpha_\mathbf{m}=\left\{i\in\mathbb{N}\colon i\geqslant-\mathbf{z}_{\alpha}+1\right\}$$ i $(j)\alpha_\mathbf{m}=j+\mathbf{z}_{\alpha}$ для всіх $j\in \operatorname{dom}\alpha_\mathbf{m}$.
\end{itemize}
Очевидно, що так визначене часткове ві\-доб\-ра\-жен\-ня $\alpha_\mathbf{m}\colon\mathbb{N}\rightharpoonup\mathbb{N}$ є частковою ізометрією, а отже $\alpha_\mathbf{m}\in\mathbf{I}\mathbb{N}_{\infty}$. Тоді з твердження~\ref{proposition-7} випливає, що $\alpha\mathfrak{C}_{\mathbf{mg}}\alpha_\mathbf{m}$, а з твердження~\ref{proposition-5}, що виконується відношення $\alpha\preccurlyeq\alpha_\mathbf{m}$.

Зауважимо, що за виконання умов:
\begin{itemize}
  \item[$(i)$] якщо $\mathbf{z}_{\alpha}<0$, то $\operatorname{ran}\alpha_\mathbf{m}=\mathbb{N}$;

  \item[$(ii)$] якщо $\mathbf{z}_{\alpha}=0$, то $\operatorname{dom}\alpha_\mathbf{m}=\operatorname{ran}\alpha_\mathbf{m}=\mathbb{N}$;

  \item[$(iii)$] якщо $\mathbf{z}_{\alpha}>0$, то $\operatorname{dom}\alpha_\mathbf{m}=\mathbb{N}$,
\end{itemize}
з того, що $\alpha_\mathbf{m}\preccurlyeq\beta$ для деяких $\alpha,\beta\in\mathbf{I}\mathbb{N}_{\infty}$, з твердження~\ref{proposition-5} випливає рівність $\alpha_\mathbf{m}=\beta$. Таким чином, $\alpha_\mathbf{m}$~--- найбільший елемент $\mathfrak{C}_{\textsf{mg}}$-класу елемента $\alpha$ стосовно природного част\-ко\-во\-го порядку $\preccurlyeq$ на напівгрупі $\mathbf{I}\mathbb{N}_{\infty}$.
\end{proof}

Нагадаємо (див. наприклад [7, \S1.12]), що \emph{біциклічною напівгрупою} (або \emph{біциклічним моноїдом}) ${\mathscr{C}}(p,q)$ називається напівгрупа з одиницею, породжена двоелементною множиною $\{p,q\}$ і визначена одним визначальним співвідношенням $pq=1$. Біциклічна напівгрупа відіграє важливу роль в теорії
напівгруп. Так, зокрема, класична теорема Олафа Андерсена [4] стверджує, що ($0$-)проста напівгрупа з ідемпотентом є цілком ($0$-)простою тоді і лише тоді, коли вона не містить ізоморфну копію біциклічної напівгрупи. Стабільні напівгрупи не містять ізоморфної копії біциклічного моноїда [5].

\begin{remark}\label{remark-10}
1. Добре відомо (див. [7, \S1.12]), що біциклічний моноїд ${\mathscr{C}}(p,q)$ ізо\-морф\-ний напівгрупі $\mathscr{C}_{\mathbb{N}}$, по\-родже\-ній частковими перетвореннями $\alpha$ та $\beta$ множини натуральних чисел $\mathbb{N}$, які визначабться на\-ступ\-ним чином:
\begin{equation*}
\operatorname{dom}\alpha=\mathbb{N}, \quad \operatorname{ran}\alpha=\mathbb{N}\setminus\{1\}, \quad (n)\alpha=n+1
\end{equation*}
i
\begin{equation*}
\operatorname{dom}\beta=\mathbb{N}\setminus\{1\}, \quad \operatorname{ran}\beta=\mathbb{N}, \quad (n)\beta=n-1.
\end{equation*}
Оскільки композиція $\alpha\beta$~--- тотожне відображення множини натуральних чисел, то з результатів про бі\-цик\-ліч\-ний моноїд, отриманих у [7, \S1.12] випливає, що кожен елемент напівгрупи $\mathscr{C}_{\mathbb{N}}$ однозначно зображається у вигляді $\beta^i\alpha^j$, де $i$ та $j$~--- деякі невід'ємні цілі числа, $\beta^0=\alpha^0$~--- тотожне перетворення множини натуральних чисел, а також ізоморфізм $\mathfrak{I}\colon \mathscr{C}_{\mathbb{N}}\to {\mathscr{C}}(p,q)$ визначається за формулою $(\beta^i\alpha^j)\mathfrak{I}=q^ip^j$. Таким чином, напівгрупа $\mathbf{I}\mathbb{N}_{\infty}$ містить ізоморфну копію бі\-цикліч\-ної напівгрупи.

2. Легко бачити, що для довільного елемента $\beta^i\alpha^j$ напівгрупи $\mathscr{C}_{\mathbb{N}}$, де $i$ та $j$~--- деякі невід'ємні цілі числа, виконуються такі умови:
\begin{itemize}
  \item[(1)] $\operatorname{dom}(\beta^i\alpha^j)=\mathbb{N}\setminus\{1,\ldots,i\}$,
  \item[(2)] $\operatorname{ran}(\beta^i\alpha^j)=\mathbb{N}\setminus\{1,\ldots,j\}$,
  \item[(3)] $(n)\beta^i\alpha^j=n-j+i$,  для $n\in\operatorname{dom}(\beta^i\alpha^j)$.
\end{itemize}
Також очевидно, що кожен частковий зсув $\mu\colon \mathbb{N}\rightharpoonup \mathbb{N}$, $n\mapsto n+k$ нескінченного променя $\left\{l,l+1,l+2,\ldots\right\}$ множини натуральних чисел $\mathbb{N}$ збігається з частковим перетворенням $\beta^{l-1}\alpha^{k+l-1}$, яке є елементом напівгрупи $\mathscr{C}_{\mathbb{N}}$.
\end{remark}

З нижче викладеного прикладу випливає, що на напівгрупі $\mathbf{I}\mathbb{N}_{\infty}$ існують конгруенції, які не є груповими.

\begin{example}\label{example-11}
Означимо відображення $\mathfrak{H}\colon\mathbf{I}\mathbb{N}_{\infty}\rightarrow\mathscr{C}_{\mathbb{N}}$ наступним чином. Нехай $\eta$~--- довільний елемент напівгрупи $\mathbf{I}\mathbb{N}_{\infty}$. Оскільки за лемою~\ref{lemma-1} кожен елемент напівгрупи $\mathbf{I}\mathbb{N}_{\infty}$ є звуженням часткового зсуву множини натуральних чисел, то існує найменше натуральне число $n_{\eta}^{\mathbf{d}}\in\operatorname{dom}\eta$ таке, що $n\in\operatorname{dom}\eta$ для всіх натуральних $n\geqslant n_{\eta}^{\mathbf{d}}$ та існує ціле число $\mathbf{z}_{\eta}$ таке, що $(i)\eta=i+\mathbf{z}_{\eta}$ для довільних $i\in\operatorname{dom}\eta$. Покладемо $(\eta)\mathfrak{H}=\overline{\eta}$~--- звуження часткового перетворення $\eta$ множини натуральних чисел на множину $\left\{i\in\mathbb{N}\colon i\geqslant n_{\eta}^{\mathbf{d}}\right\}$. Тоді із зауваження~\ref{remark-10}(2) випливає, що часткове перетворення $\overline{\eta}$ збігається з частковим перетворенням $\beta^{n_{\eta}^{\mathbf{d}}-1}\alpha^{\mathbf{z}_{\eta}+n_{\eta}^{\mathbf{d}}-1}$, яке є елементом напівгрупи $\mathscr{C}_{\mathbb{N}}$.
\end{example}

\begin{proposition}\label{proposition-12}
Відображення $\mathfrak{H}\colon\mathbf{I}\mathbb{N}_{\infty}\rightarrow\mathscr{C}_{\mathbb{N}}$ є сюр'єктивним гомоморфізмом моноїдів.
\end{proposition}

\begin{proof}
Спочатку зауважимо, що з означення відображення $\mathfrak{H}\colon\mathbf{I}\mathbb{N}_{\infty}\rightarrow\mathscr{C}_{\mathbb{N}}$ випливає, що $(\beta^i\alpha^j)\mathfrak{H}=\beta^i\alpha^j$ для довільного елемента $\beta^i\alpha^j$ напівгрупи $\mathscr{C}_{\mathbb{N}}$.

Нехай $\eta$ та $\mu$~--- довільні елементи напівгрупи $\mathbf{I}\mathbb{N}_{\infty}$. Оскільки за лемою~\ref{lemma-1} кожен елемент напівгрупи $\mathbf{I}\mathbb{N}_{\infty}$ є звуженням часткового зсуву множини натуральних чисел, то існують найменші натуральні числа $n_{\eta}^{\mathbf{d}}\in\operatorname{dom}\eta$ і $n_{\mu}^{\mathbf{d}}\in\operatorname{dom}\mu$ такі, що $n\in\operatorname{dom}\eta$ та $m\in\operatorname{dom}\mu$ для всіх натуральних $n\geqslant n_{\eta}^{\mathbf{d}}$ і $m\geqslant n_{\mu}^{\mathbf{d}}$, та існують цілі числа $\mathbf{z}_{\eta}$ і $\mathbf{z}_{\mu}$ такі, що $(i)\eta=i+\mathbf{z}_{\eta}$ і $(j)\mu=j+\mathbf{z}_{\mu}$ для довільних $i\in\operatorname{dom}\eta$ та $j\in\operatorname{dom}\mu$.

З означення відображення $\mathfrak{H}\colon\mathbf{I}\mathbb{N}_{\infty}\rightarrow\mathscr{C}_{\mathbb{N}}$, використавши зауваження~\ref{remark-10}(1), отримуємо:
\begin{equation*}
\begin{split}
    &(\eta)\mathfrak{H}(\mu)\mathfrak{H}
= \beta^{n_{\eta}^{\mathbf{d}}-1}\alpha^{\mathbf{z}_{\eta}+n_{\eta}^{\mathbf{d}}-1} \beta^{n_{\mu}^{\mathbf{d}}-1}\alpha^{\mathbf{z}_{\mu}+n_{\mu}^{\mathbf{d}}-1}=\\
& =
\left\{
  \begin{array}{ll}
    \beta^{n_{\eta}^{\mathbf{d}}-1}\alpha^{\mathbf{z}_{\eta}+\mathbf{z}_{\mu}+n_{\eta}^{\mathbf{d}}-1}, & \hbox{якщо~} \mathbf{z}_{\eta}+n_{\eta}^{\mathbf{d}}>n_{\mu}^{\mathbf{d}};\\
    \beta^{n_{\eta}^{\mathbf{d}}-1}\alpha^{\mathbf{z}_{\mu}+n_{\mu}^{\mathbf{d}}-1}, & \hbox{якщо~} \mathbf{z}_{\eta}+n_{\eta}^{\mathbf{d}}=n_{\mu}^{\mathbf{d}};\\
    \beta^{n_{\mu}^{\mathbf{d}}-\mathbf{z}_{\eta}-1}\alpha^{\mathbf{z}_{\mu}+n_{\mu}^{\mathbf{d}}-1}, & \hbox{якщо~} \mathbf{z}_{\eta}+n_{\eta}^{\mathbf{d}}<n_{\mu}^{\mathbf{d}}.
  \end{array}
\right.
\end{split}
\end{equation*}
Оскільки за лемою~\ref{lemma-1} кожен елемент напівгрупи $\mathbf{I}\mathbb{N}_{\infty}$ є звуженням часткового зсуву множини натуральних чисел, то $\mathbf{z}_{\eta\mu}=\mathbf{z}_{\eta}+\mathbf{z}_{\mu}$. Далі визначимо найменше натуральне число $n_{\eta\mu}^{\mathbf{d}}\in\operatorname{dom}\eta$ таке, що $n\in\operatorname{dom}(\eta\mu)$ для всіх натуральних $n\geqslant n_{\eta\mu}^{\mathbf{d}}$. Знову, оскільки за лемою~\ref{lemma-1} кожен елемент напівгрупи $\mathbf{I}\mathbb{N}_{\infty}$ є звуженням часткового зсуву множини натуральних чисел, то отримуємо, що
\begin{equation*}
  n_{\eta\mu}^{\mathbf{d}}=\max\left\{\mathbf{z}_{\eta}+n_{\eta}^{\mathbf{d}},n_{\mu}^{\mathbf{d}}\right\}-\mathbf{z}_{\eta}.
\end{equation*}
Тоді
\begin{equation*}
\begin{split}
  &(\eta\mu)\mathfrak{H} =\beta^{n_{\eta\mu}^{\mathbf{d}}-1}\alpha^{\mathbf{z}_{\eta\mu}+n_{\eta\mu}^{\mathbf{d}}-1}= \\
    =&\beta^{\max\{\mathbf{z}_{\eta}{+}n_{\eta}^{\mathbf{d}},n_{\mu}^{\mathbf{d}}\}{-}\mathbf{z}_{\eta}{-}1} \alpha^{\mathbf{z}_{\eta}{+}\mathbf{z}_{\mu}{+}\!\max\{\mathbf{z}_{\eta}{+}n_{\eta}^{\mathbf{d}},n_{\mu}^{\mathbf{d}}\}{-}\mathbf{z}_{\eta}{-}1}\! \\
    =&\beta^{\max\left\{\mathbf{z}_{\eta}+n_{\eta}^{\mathbf{d}},n_{\mu}^{\mathbf{d}}\right\}-\mathbf{z}_{\eta}-1} \alpha^{\mathbf{z}_{\mu}+\max\left\{\mathbf{z}_{\eta}+n_{\eta}^{\mathbf{d}},n_{\mu}^{\mathbf{d}}\right\}-1}= \\
    =&
\left\{
  \begin{array}{ll}
    \beta^{n_{\eta}^{\mathbf{d}}-1}\alpha^{\mathbf{z}_{\eta}+\mathbf{z}_{\mu}+n_{\eta}^{\mathbf{d}}-1}, & \hbox{якщо~} \mathbf{z}_{\eta}+n_{\eta}^{\mathbf{d}}>n_{\mu}^{\mathbf{d}};\\
    \beta^{n_{\eta}^{\mathbf{d}}-1}\alpha^{\mathbf{z}_{\mu}+n_{\mu}^{\mathbf{d}}-1}, & \hbox{якщо~} \mathbf{z}_{\eta}+n_{\eta}^{\mathbf{d}}=n_{\mu}^{\mathbf{d}};\\
    \beta^{n_{\mu}^{\mathbf{d}}-\mathbf{z}_{\eta}-1}\alpha^{\mathbf{z}_{\mu}+n_{\mu}^{\mathbf{d}}-1}, & \hbox{якщо~} \mathbf{z}_{\eta}+n_{\eta}^{\mathbf{d}}<n_{\mu}^{\mathbf{d}}.
  \end{array}
\right.
\end{split}
\end{equation*}
Таким чином, виконується рівність $(\eta\mu)\mathfrak{H}=(\eta)\mathfrak{H}(\mu)\mathfrak{H}$ для всіх елементів $\eta$ і $\mu$ напівгрупи $\mathbf{I}\mathbb{N}_{\infty}$, а отже відображення $\mathfrak{H}\colon\mathbf{I}\mathbb{N}_{\infty}\rightarrow\mathscr{C}_{\mathbb{N}}$ є гомоморфізмом моноїдів.
\end{proof}

Очевидно, що ядро
\begin{equation*}
  \ker \mathfrak{H}=\left\{\left(\eta,\mu\right)\in\mathbf{I}\mathbb{N}_{\infty}\times \mathbf{I}\mathbb{N}_{\infty}\colon (\eta)\mathfrak{H}=(\mu)\mathfrak{H}\right\},
\end{equation*}
вище означеного гомоморфізму $\mathfrak{H}\colon\mathbf{I}\mathbb{N}_{\infty}\rightarrow\mathscr{C}_{\mathbb{N}}$ є конгруенцією на напівгрупі $\mathbf{I}\mathbb{N}_{\infty}$, яка не є груповою.

\smallskip

\emph{Гомоморфною ретракцією} називається відображення з напівгрупи $S$ в $S$, яке є одночасно ретракцією та гомоморфізмом [7]. Образ напівгрупи $S$ при її гомоморфній ретракції називається \emph{гомоморфним рет\-рак\-том}. Тобто гомоморфний ретракт напівгрупи $S$~--- це така піднапівгрупа $T$ в $S$, що існує гомоморфізм з $S$ в $S$, для якого піднапівгрупа $T$ є множиною всіх його нерухомих точок.

\smallskip

З твердження~\ref{proposition-12} випливає

\begin{corollary}\label{corollary-13}
Напівгрупа $\mathscr{C}_{\mathbb{N}}$  є гомоморф\-ним ретрактом напівгрупи $\mathbf{I}\mathbb{N}_{\infty}$.
\end{corollary}

\begin{remark}\label{remark-14}
Нехай $S$~--- напівгрупа, $T$~--- піднапівгрупа напівгрупи $S$ i $\mathfrak{C}_S$~--- конгруенція на $S$. Тоді з означення поняття конгруенція випливає, що звуження $\mathfrak{C}_T=\mathfrak{C}_S|_{T\times T}$ відношення $\mathfrak{C}_S$ на декартовий добуток $T\times T$ є конгруенцією на напівгрупі $T$
\end{remark}

Виявляється, що піднапівгрупа $\mathscr{C}_{\mathbb{N}}$ напівгрупи $\mathbf{I}\mathbb{N}_{\infty}$ дає можливість отримати критерій, коли конгруенція на $\mathbf{I}\mathbb{N}_{\infty}$ є груповою.

\begin{theorem}\label{theorem-15}
Конгруенція $\mathfrak{C}$ на напівгрупі $\mathbf{I}\mathbb{N}_{\infty}$ є груповою тоді і лише тоді, коли її звуження $\mathfrak{C}_{\mathscr{C}_{\mathbb{N}}}=\mathfrak{C}|_{\mathscr{C}_{\mathbb{N}}\times\mathscr{C}_{\mathbb{N}}}$ на піднапівгрупу $\mathscr{C}_{\mathbb{N}}$ не є тотожною конгруенцією на $\mathscr{C}_{\mathbb{N}}$.
\end{theorem}

\begin{proof}
$(\Rightarrow)$ Із зауваження~\ref{remark-14} випливає, якщо $\mathfrak{C}$~--- групова на напівгрупі $\mathbf{I}\mathbb{N}_{\infty}$, то її звуження $\mathfrak{C}_{\mathscr{C}_{\mathbb{N}}}$ на піднапівгрупу $\mathscr{C}_{\mathbb{N}}$ є також конгруенцією, а оскільки всі ідемпотенти напівгрупи $\mathbf{I}\mathbb{N}_{\infty}$ є $\mathfrak{C}$-еквівалентними, то всі ідемпотенти піднапівгрупи $\mathscr{C}_{\mathbb{N}}$ є також $\mathfrak{C}_{\mathscr{C}_{\mathbb{N}}}$-еквівалентними. За за\-ува\-жен\-ням~\ref{remark-10}(1) напівгрупа $\mathscr{C}_{\mathbb{N}}$ ізоморфна біциклічному моноїдові, то з наслідку~1.32 [7] випливає, що конгруенція $\mathfrak{C}_{\mathscr{C}_{\mathbb{N}}}$ на $\mathscr{C}_{\mathbb{N}}$ є також груповою, а отже $\mathfrak{C}_{\mathscr{C}_{\mathbb{N}}}$ не є відношенням рівності на $\mathscr{C}_{\mathbb{N}}$.

$(\Leftarrow)$ Припустимо, що конгруенція $\mathscr{C}$ на напівгрупі $\mathbf{I}\mathbb{N}_{\infty}$ є такою, що її звуження $\mathfrak{C}_{\mathscr{C}_{\mathbb{N}}}$ на піднапівгрупу $\mathscr{C}_{\mathbb{N}}$ не є відношенням рівності на $\mathscr{C}_{\mathbb{N}}$. Оскільки за за\-ува\-жен\-ням~\ref{remark-10}(1) напівгрупа $\mathscr{C}_{\mathbb{N}}$ ізоморфна біциклічному моноїдові, то з наслідку~1.32 [7] випливає, що конгруенція $\mathfrak{C}_{\mathscr{C}_{\mathbb{N}}}$ на $\mathscr{C}_{\mathbb{N}}$ є груповою, а отже всі ідемпотенти напівгрупи $\mathscr{C}_{\mathbb{N}}$ є $\mathfrak{C}_{\mathscr{C}_{\mathbb{N}}}$-еквівалентними.

Нехай $\varepsilon$~--- довільний ідемпотент напівгрупи $\mathbf{I}\mathbb{N}_{\infty}$. Оскільки $\varepsilon$ є тотожним ві\-доб\-ра\-жен\-ням коскінченної підмножини  $\operatorname{dom}\varepsilon$ множини натуральних чисел і $(\mathbb{N},\leqslant)$~--- цілком впорядкована множина, то існує найменше натуральне число $n_{\varepsilon}^{\mathbf{d}}\in\operatorname{dom}\varepsilon$ таке, що $n\in\operatorname{dom}\varepsilon$ для всіх натуральних $n\geqslant n_{\varepsilon}^{\mathbf{d}}$. Нехай $\overline{\varepsilon}$~--- тотожне відображення множини $\left\{n\in\mathbb{N}\colon n\geqslant n_{\varepsilon}^{\mathbf{d}}+1\right\}$. Тоді очевидно, що $\overline{\varepsilon}$~--- ідемпотент підмоноїда $\mathscr{C}_{\mathbb{N}}$ моноїда $\mathbf{I}\mathbb{N}_{\infty}$, а оскільки конгруенція $\mathfrak{C}_{\mathscr{C}_{\mathbb{N}}}$ на $\mathscr{C}_{\mathbb{N}}$ є груповою, то $\mathbb{I}\mathfrak{C}_{\mathscr{C}_{\mathbb{N}}}\overline{\varepsilon}$, а отже й $\mathbb{I}\mathfrak{C}\overline{\varepsilon}$. Також, легко бачити, що $\overline{\varepsilon}\preccurlyeq\varepsilon\preccurlyeq\mathbb{I}$, де $\preccurlyeq$~--- природний частковий порядок на напівгрупі $\mathbf{I}\mathbb{N}_{\infty}$, і тоді з відношення $\mathbb{I}\mathfrak{C}\overline{\varepsilon}$ випливає, що
$
  \varepsilon=(\varepsilon\mathbb{I})\mathfrak{C}(\varepsilon\overline{\varepsilon})=\overline{\varepsilon},
$ 
а отже $\varepsilon\mathfrak{C}\mathbb{I}$. З довільності вибору ідемпотента $\varepsilon$ в $\mathbf{I}\mathbb{N}_{\infty}$ випливає, що всі ідемпотенти напівгрупи $\mathbf{I}\mathbb{N}_{\infty}$ є $\mathscr{C}$-еквівалентними. Тоді з твердження~1.4.21(3) з монографії [9] випливає, що $\mathscr{C}$ --- групова конгруенція на напівгрупі $\mathbf{I}\mathbb{N}_{\infty}$.
\end{proof}

З теореми~\ref{theorem-15} випливає такий наслідок:

\begin{corollary}\label{corollary-16}
Для довільної конгруенції $\mathfrak{C}$ на моноїді $\mathbf{I}\mathbb{N}_{\infty}$ виконується лише одна з умов:
\begin{enumerate}
  \item[$(1)$] $\mathfrak{C}$ --- групова конгруенція на  $\mathbf{I}\mathbb{N}_{\infty}$;
  \item[$(2)$] звуження природного гомоморфізму $\mathfrak{C}^{\sharp}\colon \mathbf{I}\mathbb{N}_{\infty}\to \mathbf{I}\mathbb{N}_{\infty}/\mathfrak{C}$ на підмоноїд $\mathscr{C}_{\mathbb{N}}$ є тотожним відображенням.
\end{enumerate}
\end{corollary}

Також з теорем~\ref{theorem-8} i~\ref{theorem-15}  випливає наслідок~\ref{corollary-17}.

\begin{corollary}\label{corollary-17}
Нехай напівгрупа не міс\-тить ізоморфної копії біциклічного моноїда. Тоді для довільного гомоморфізму $\mathfrak{F}\colon\mathbf{I}\mathbb{N}_{\infty}\to S$ існує єдиний гомоморфізм $\mathfrak{H}\colon(\mathbb{Z},+)\to S$ такий, що наступна діаграма
\begin{equation*}
\xymatrix{
\mathbf{I}\mathbb{N}_{\infty}\ar[rr]^{\mathfrak{F}}\ar[dd]_{\mathfrak{C}_{\mathbf{mg}}^{\sharp}} && S\\
&&\\
(\mathbb{Z},+)\ar[rruu]_{\mathfrak{H}}
}
\end{equation*}
є комутативною.
\end{corollary}

\begin{lemma}\label{lemma-18}
Нехай $\mathfrak{C}$ ---  конгруенція на напівгрупі $\mathbf{I}\mathbb{N}_{\infty}$, яка відмінна від групової та тотожної. Тоді існують два різні $\mathfrak{C}$-еквівалентні ідемпотенти $\varepsilon$ i $\iota$ напівгрупи $\mathbf{I}\mathbb{N}_{\infty}$ такі, що $n_{\varepsilon}^{\mathbf{d}}=n_{\iota}^{\mathbf{d}}$.
\end{lemma}

\begin{proof}
Оскільки конгруенція $\mathfrak{C}$ на напівгрупі $\mathbf{I}\mathbb{N}_{\infty}$ відмінна від тотожної, то існують два різні $\mathfrak{C}$-еквівалентні елементи $\gamma,\delta\in \mathbf{I}\mathbb{N}_{\infty}$. За твердженням~\ref{proposition-3} кожен $\mathscr{H}$-клас напівгрупи $\mathbf{I}\mathbb{N}_{\infty}$ є одноелементним, то виконується хоча б одна з умов: $\gamma\gamma^{-1}\neq\delta\delta^{-1}$  або $\gamma^{-1}\gamma\neq\delta^{-1}\delta$. Отже, існують два різні $\mathfrak{C}$-еквівалентні ідемпотенти $\varepsilon,\iota\in \mathbf{I}\mathbb{N}_{\infty}$. За теоремою~\ref{theorem-15} елементи $\varepsilon$ i $\iota$ не можуть одно\-час\-но бути ідемпотентами підмоноїда $\mathscr{C}_{\mathbb{N}}$ моноїда $\mathbf{I}\mathbb{N}_{\infty}$, оскільки конгруенція $\mathfrak{C}$ на $\mathbf{I}\mathbb{N}_{\infty}$ не є груповою.

Припустимо, що $n_{\varepsilon}^{\mathbf{d}}>n_{\iota}^{\mathbf{d}}$. Нехай $\overline{\varepsilon}$~--- тотожне ві\-доб\-ра\-жен\-ня множини $\left\{k\in\mathbb{N}\colon k\geqslant n_{\iota}^{\mathbf{d}}-1\right\}$ i $\overline{\iota}$~--- тотожне ві\-доб\-ра\-жен\-ня множини $\left\{k\in\mathbb{N}\colon k\geqslant n_{\iota}^{\mathbf{d}}\right\}$. Тоді, очевидно, що $\overline{\varepsilon},\overline{\iota}\in E(\mathscr{C}_{\mathbb{N}})$ i  $\overline{\varepsilon}\neq\overline{\iota}$. Позаяк $\varepsilon\mathfrak{C}\iota$, то $\overline{\varepsilon}=(\varepsilon\overline{\varepsilon})\mathfrak{C}(\iota\overline{\varepsilon})=\overline{\iota}$, а отже два різні ідемпотенти підмоноїда $\mathscr{C}_{\mathbb{N}}$ моноїда $\mathbf{I}\mathbb{N}_{\infty}$ є $\mathfrak{C}$-еквівалентними. Тоді за теоремою~\ref{theorem-15} конгруенція $\mathfrak{C}$ є груповою, що суперечить припущенню. З отриманого протиріччя випливає, що нерівність $n_{\varepsilon}^{\mathbf{d}}>n_{\iota}^{\mathbf{d}}$ не виконується для двох різних $\mathfrak{C}$-еквівалентних ідемпотентів $\varepsilon$ i $\iota$ напівгрупи $\mathbf{I}\mathbb{N}_{\infty}$.

Аналогічно доводиться, що нерівність $n_{\varepsilon}^{\mathbf{d}}<n_{\iota}^{\mathbf{d}}$ не виконується для двох різних $\mathfrak{C}$-еквівалентних ідемпотентів $\varepsilon$ i $\iota$ напівгрупи $\mathbf{I}\mathbb{N}_{\infty}$.
\end{proof}

\begin{lemma}\label{lemma-19}
Для елемента $\xi$ моноїда $\mathbf{I}\mathbb{N}_{\infty}$ наступні умови є еквівалентними:
\begin{enumerate}
  \item[$(1)$] $\xi\in\mathscr{C}_{\mathbb{N}}$;
  \item[$(2)$] $n_{\xi}^{\mathbf{d}}=\min\{n\in \mathbb{N}\colon n\in\operatorname{dom}\xi\}$;
  \item[$(3)$] $n_{\xi}^{\mathbf{r}}=\min\{n\in \mathbb{N}\colon n\in\operatorname{ran}\xi\}$.
\end{enumerate}
\end{lemma}

\begin{proof}
Імплікації $(1)\Rightarrow(2)$ та $(1)\Rightarrow(3)$ випливають з означення напівгрупи $\mathscr{C}_{\mathbb{N}}$ (див. за\-ува\-жен\-ня~\ref{remark-10}(1).

Імплікація $(2)\Rightarrow(1)$ випливає із за\-ува\-жен\-ня~\ref{remark-10}(2). Справді, за лемою~\ref{lemma-1} елемент $\xi$ є звуженням часткового зсуву множини натуральних чисел на коскінченну підмножину в $\mathbb{N}$, а отже маємо. що $(n)\xi=n-n_{\xi}^{\mathbf{d}}+n_{\xi}^{\mathbf{r}}$ для всіх $n\in\operatorname{dom}\xi$. Тоді із за\-ува\-жен\-ня~\ref{remark-10}(2) випливає, що $\xi=\beta^{n_{\xi}^{\mathbf{d}}-1}\alpha^{n_{\xi}^{\mathbf{r}}-1}$.

Еквівалентність умов (2) та (3) є наслідком леми~~\ref{lemma-1}, оскільки елемент $\xi$ є звуженням част\-кового зсуву множини натуральних чисел на коскінченну підмножину в $\mathbb{N}$.
\end{proof}

\begin{definition}\label{definition-20}
Нехай $\xi$~--- довільний елемент напівгрупи $\mathbf{I}\mathbb{N}_{\infty}$ такий, що $\xi\notin \mathscr{C}_{\mathbb{N}}$. Означимо
\begin{equation*}
\begin{split}
 & \underline{n}_{\xi}^{\mathbf{d}}=\min\left\{n\in \mathbb{N}\colon n\in\operatorname{dom}\xi\right\}, \\
 & \underline{n}_{\xi}^{\mathbf{r}}=\min\left\{n\in \mathbb{N}\colon n\in\operatorname{ran}\xi\right\},\\
 & \overline{n}_{\xi}^{\mathbf{d}}=\max\left\{n\in\operatorname{dom}\xi\colon n<n_{\xi}^{\mathbf{d}}\right\},\\
 & \overline{n}_{\xi}^{\mathbf{r}}=\max\left\{n\in\operatorname{ran}\xi\colon n<n_{\xi}^{\mathbf{r}}\right\}.
\end{split}
\end{equation*}
\end{definition}

Очевидно, що виконуються наступні умови:
$\underline{n}_{\xi}^{\mathbf{d}}\leqslant\overline{n}_{\xi}^{\mathbf{d}}<n_{\xi}^{\mathbf{d}}$, $\underline{n}_{\xi}^{\mathbf{r}}\leqslant\overline{n}_{\xi}^{\mathbf{r}}<n_{\xi}^{\mathbf{r}}$,
$(\underline{n}_{\xi}^{\mathbf{d}})\xi=\underline{n}_{\xi}^{\mathbf{r}}$ i
$(\overline{n}_{\xi}^{\mathbf{d}})\xi=\overline{n}_{\xi}^{\mathbf{r}}$,
для довільного $\xi\in\mathbf{I}\mathbb{N}_{\infty}\setminus\mathscr{C}_{\mathbb{N}}$. Також, елемент $\xi\in\mathbf{I}\mathbb{N}_{\infty}\setminus\mathscr{C}_{\mathbb{N}}$ є ідемпотентом тоді і лише тоді, коли $\underline{n}_{\xi}^{\mathbf{d}}=\underline{n}_{\xi}^{\mathbf{r}}$ i $\overline{n}_{\xi}^{\mathbf{d}}=\overline{n}_{\xi}^{\mathbf{r}}$.

\begin{lemma}\label{lemma-21}
Кожен ідемпотент напівгрупи $\mathbf{I}\mathbb{N}_{\infty}$ є одиницею піднапівгрупи в $\mathbf{I}\mathbb{N}_{\infty}$, яка ізоморфна біциклічній напівгрупі.
\end{lemma}

\begin{proof}
Нехай $\varepsilon$~--- довільний ідемпотент напівгрупи $\mathbf{I}\mathbb{N}_{\infty}$. Якщо $\varepsilon\in\mathscr{C}_{\mathbb{N}}$ і оскільки напівгрупа $\mathscr{C}_{\mathbb{N}}$ ізоморфна біциклічній напівгрупі, то не зменшуючи загальності можемо вважати, що ідемпотент $\varepsilon$ можна ототожнити з ідемпотентом $q^ip^i$ біциклічного моноїда $\mathscr{C}(p,q)$, для деякого невід'ємного цілого чис\-ла $i$.  Тоді
\begin{equation*}
  \begin{array}{l}
     q^ip^{i+1}\cdot q^{i+1}p^{i}=q^ip^{i},\\
     q^{i+1}p^{i}\cdot q^{i}p^{i+1}=q^{i+1}p^{i+1},\\
     q^ip^{i+1}\cdot q^{i}p^{i}=q^ip^{i+1},\\
     q^{i}p^{i}\cdot q^{i}p^{i+1}=q^{i}p^{i+1},\\
     q^{i+1}p^{i}\cdot q^{i}p^{i}=q^{i+1}p^{i},\\
     q^{i}p^{i}\cdot q^{i+1}p^{i}=q^{i+1}p^{i}
   \end{array}
\end{equation*}
i $q^{i+1}p^{i+1}\neq q^{i}p^{i}$, а отже за лемою~1.31 з [7] піднапівгрупа в $\mathbf{I}\mathbb{N}_{\infty}$, породжена елементами $q^ip^{i+1}$ i $q^{i+1}p^{i}$ ізоморфна біциклічній напівгрупі.

Припустимо, що $\varepsilon\notin\mathscr{C}_{\mathbb{N}}$. Тоді для ідемпотента $\varepsilon$ виконуються умови $\underline{n}_{\varepsilon}^{\mathbf{d}}\leqslant \overline{n}_{\varepsilon}^{\mathbf{d}}<{n}_{\varepsilon}^{\mathbf{d}}$. Означимо часткову бієкцію $\gamma\colon \mathbb{N}\rightharpoonup\mathbb{N}$ так:
$\operatorname{dom}\gamma=\operatorname{dom}\varepsilon$, $$\operatorname{ran}\gamma=\left\{i-\underline{n}_{\varepsilon}^{\mathbf{d}}+ \overline{n}_{\varepsilon}^{\mathbf{d}}\colon i\in\operatorname{dom}\varepsilon \right\}$$ i
$(n)\gamma=n-\underline{n}_{\varepsilon}^{\mathbf{d}}+ \overline{n}_{\varepsilon}^{\mathbf{d}}$,  для всіх $n\in \operatorname{dom}\gamma$.
Тоді, очевидно, що $\gamma\in\mathbf{I}\mathbb{N}_{\infty}$ і виконуються такі співвідношення:
$$
\varepsilon\gamma=\gamma\varepsilon=\gamma,
$$
$$
\gamma^{-1}\varepsilon=\gamma^{-1}\varepsilon^{-1}=(\varepsilon\gamma)^{-1}=\gamma^{-1},
$$
$$
\varepsilon\gamma^{-1}=\varepsilon^{-1}\gamma^{-1}=(\gamma\varepsilon)^{-1}=\gamma^{-1},
$$
$$
\gamma\gamma^{-1}=\varepsilon \qquad  \hbox{i}\qquad \gamma^{-1}\gamma\neq\varepsilon,
$$
а отже за лемою~1.31 з [7] піднапівгрупа в $\mathbf{I}\mathbb{N}_{\infty}$, породжена елементами $\gamma$ i $\gamma^{-1}$ ізо\-морф\-на біциклічній напівгрупі.
\end{proof}



\begin{theorem}\label{theorem-23}
Для конгруенції $\mathfrak{C}$ на напівгрупі $\mathbf{I}\mathbb{N}_{\infty}$ наступні умови є еквівалент\-ними:
\begin{enumerate}
  \item[$(1)$] $\mathfrak{C}$~--- групова конгруенція на $\mathbf{I}\mathbb{N}_{\infty}$;

  \item[$(3)$] існує піднапівгрупа $S$ в $\mathbf{I}\mathbb{N}_{\infty}$, яка ізоморфна біциклічній напівгрупі та два різні елементи напівгрупи $S$ є $\mathfrak{C}$-ек\-вівалентними;

  \item[$(3)$] для довільної піднапівгрупи $T$ в $\mathbf{I}\mathbb{N}_{\infty}$, яка ізоморфна біциклічній напівгрупі, два різні елементи напівгрупи $T$ є $\mathfrak{C}$-ек\-вівалентними.
\end{enumerate}
\end{theorem}

\begin{proof}
Імплікація $(1)\Rightarrow(2)$ випливає з теореми~\ref{theorem-15}, а імплікації $(1)\Rightarrow(3)$ і $(3)\Rightarrow(2)$ є очевидними.

Доведемо, що виконується імплікація $(2)\Rightarrow(1)$. Припустимо, що $S$~--- піднапівгрупа в $\mathbf{I}\mathbb{N}_{\infty}$, яка ізоморфна біциклічній напівгрупі та два різні елементи  напівгрупи $S$ є $\mathfrak{C}$-еквівалентними. Тоді за наслідком~1.32 з [7] усі ідемпотенти напівгрупи $S$ є $\mathfrak{C}$-еквівалентними.

Нехай $\varepsilon$~--- одиниця напівгрупи $S$. З теореми~\ref{theorem-15} випливає, що не зменшуючи загальності, можемо вважати, що $\varepsilon\notin\mathscr{C}_{\mathbb{N}}$. Справді, припустивши, що $\varepsilon\in\mathscr{C}_{\mathbb{N}}$, то, оскільки за лемою~\ref{lemma-1} кожен елемент напівгрупи $\mathbf{I}\mathbb{N}_{\infty}$ є звуженням часткового зсуву множини натуральних чисел, існує елемент $\gamma\in S\cap \mathscr{C}_{\mathbb{N}}$ такий, що $\operatorname{dom}\varepsilon=\operatorname{dom}\gamma$ i $\gamma\neq \varepsilon$. Тоді $\gamma\gamma^{-1}=\varepsilon$, $\gamma^{-1}\gamma\neq\varepsilon$ i $\gamma,\gamma^{-1},\gamma^{-1}\gamma\in\mathscr{C}_{\mathbb{N}}\cap S$. А отже, два ідемпотента напівгрупи $\mathscr{C}_{\mathbb{N}}$ є $\mathfrak{C}$-еквівалентними. Тоді за наслідком~1.32 з [7] i теоремою~\ref{theorem-15}, $\mathfrak{C}$~--- групова конгруенція на напівгрупі $\mathbf{I}\mathbb{N}_{\infty}$.

Зафіксуємо довільний елемент $\gamma$ напівгрупи $S$ такий, що $\gamma\mathscr{L}\varepsilon$. Для довільного натурального числа $i$ покладемо $$\varepsilon_i=\underbrace{\gamma^{-1}\ldots\gamma^{-1}}_{{\footnotesize\hbox{$i$-разів}}}\underbrace{\gamma\ldots\gamma}_{{\footnotesize\hbox{$i$-разів}}}.$$ Тоді, оче\-вид\-но, що $\varepsilon_i$~--- ідемпотент напівгрупи $S$ i $$\underbrace{\gamma\ldots\gamma}_{{\footnotesize\hbox{$i$-разів}}}\underbrace{\gamma^{-1}\ldots\gamma^{-1}}_{{\footnotesize\hbox{$i$-разів}}}=\varepsilon,$$ бо $\gamma\gamma^{-1}=\varepsilon$. Також, оскільки всі $\mathscr{H}$-класи в біциклічному моноїді є тривіальними та його напівгрупа ідемпотентів є $\omega$-ланцюгом, то
$$
\operatorname{dom}\varepsilon_{i+1}\subsetneqq \operatorname{dom}\varepsilon_{i}\subsetneqq \operatorname{dom}\varepsilon
$$
для довільного натурального числа $i$. З $\varepsilon\notin\mathscr{C}_{\mathbb{N}}$ випливає, що
\begin{equation*}
  \underline{n}_{\varepsilon}^{\mathbf{d}}=\underline{n}_{\varepsilon}^{\mathbf{r}}\leqslant\overline{n}_{\varepsilon}^{\mathbf{d}}= \overline{n}_{\gamma^i}^{\mathbf{d}}< {n}_{\varepsilon}^{\mathbf{d}}={n}_{\gamma^i}^{\mathbf{d}},
\end{equation*}
\begin{equation*}
\underline{n}_{\varepsilon}^{\mathbf{d}}=\underline{n}_{\varepsilon}^{\mathbf{r}}<\underline{n}_{\gamma^i}^{\mathbf{r}}\leqslant \overline{n}_{\gamma^i}^{\mathbf{r}}<{n}_{\gamma^i}^{\mathbf{r}},
\end{equation*}
\begin{equation*}
  \underline{n}_{\gamma^i}^{\mathbf{r}}<\underline{n}_{\gamma^{i+k}}^{\mathbf{r}}, \quad \overline{n}_{\gamma^i}^{\mathbf{r}}<\overline{n}_{\gamma^{i+k}}^{\mathbf{r}} \quad \hbox{i} \quad  {n}_{\gamma^i}^{\mathbf{r}}<{n}_{\gamma^{i+k}}^{\mathbf{r}},
\end{equation*}
для довільних натуральних чисел $i$ та $k$.

Тоді існує таке натуральне число $j$, що ${n}_{\gamma}^{\mathbf{r}}\leqslant\underline{n}_{\gamma^j}^{\mathbf{r}}$. Означимо: $\varphi_0,\varphi_1$ i $\varphi_2$~--- то\-тож\-ні відображення множин
$\left\{n\in\mathbb{N}\colon n\geqslant\underline{n}_{\varepsilon}^{\mathbf{d}}\right\}$,
$\left\{n\in\mathbb{N}\colon n\geqslant{n}_{\varepsilon}^{\mathbf{d}}\right\}$ i $\big\{n\in\mathbb{N}\colon n\geqslant\underline{n}_{\gamma^j}^{\mathbf{d}}\big\}$,
відповідно. Тоді $\varphi_0,\varphi_1$ i $\varphi_2$~--- різні ідемпотенти напівгрупи $\mathscr{C}_{\mathbb{N}}$, і оскільки
$$
\varepsilon_j\preccurlyeq \varphi_2\preccurlyeq \varphi_1\preccurlyeq \varepsilon\preccurlyeq \varphi_0,
$$
то з умови $\varepsilon_j\mathfrak{C}\varepsilon$
випливає, що $\varphi_2\mathfrak{C}\varphi_1$. Оскільки напівгрупа $\mathscr{C}_{\mathbb{N}}$ ізоморфна біциклічному моноїдові, то за наслідком~1.32 з [7] та теоремою~\ref{theorem-15}, $\mathfrak{C}$~--- групова конгруенція на напівгрупі $\mathbf{I}\mathbb{N}_{\infty}$.
\end{proof}
\bigskip

{\normalsize\centerline{СПИСОК ЛІТЕРАТУРИ}
\medskip

1. \emph{Безущак О.О.}
Від\-ношення Ґріна інверсної на\-півгрупи частково визначених коскінченних ізометрій дис\-крет\-ної лінійки //
Вісник Київ. ун-ту. Сер. фіз.-мат.~-- 2008.~--  N1.~-- С.~12--16.

2. \emph{Вагнер В.В.}
Обобщённые группы //
ДАН \break {С}{С}{С}{Р} -- 1952.~-- {\bf 84}.~-- С.~1119--1122.

3. \emph{Гутік О., Савчук А.}
Про напівгрупу $\mathbf{ID}_{\infty}$ //
Вісник Львів. ун-ту. Сер. мех.-мат.~-- 2017.~-- \textbf{83}.~-- С.~5--19.

4. \emph{Andersen O.}
Ein Bericht uber die Struktur abstrakter Halbgruppen. ---
PhD Thesis, Hamburg, 1952.

5. \emph{Anderson L.W., Hunter R.P., Koch R.J.}
Some results on stability in semigroups //
Trans. Amer. Math. Soc.~-- 1965.~-- \textbf{117}.~-- P.~521--529.

6. \emph{Bezushchak O.}
On growth of the inverse semigroup of partially defined co–finite automorphisms of integers //
Algebra Discrete Math.~-- 2004.~--  N2.~-- P.~45--55.

7. \emph{Clifford A. H., Preston G. B.}
The algebraic theory of semigroups. -- Providence: Amer. Math. Soc., 1961. -- Vol. 1. --- xv+224 p.; 1972. -- Vol. 2. -- xv+352~p.

8. \emph{Gutik O., Repov\v{s} D.}
Topological monoids of monotone, injective partial selfmaps of $\mathbb{N}$ having cofinite domain and image //
Stud. Sci. Math. Hungar.~-- 2011.~-- {\bf 48}, N3.~-- P.~342--353.

9. \emph{Lawson M.} Inverse semigroups. The theory of partial symmetries. -- Singapore: World Sci., 1998. -- xiii+411 p.

10. \emph{McFadden R., O’Carroll L.}
$F$-inverse semigroups //
Proc. Lond. Math. Soc., III. Ser. ~-- 1971.~-- {\bf 22}.~-- P.~652--666.

11. \emph{Petrich M.} Inverse semigroups. -- New York: John Wiley \& Sons, 1984. -- 674 p.
}

\end{document}